\documentclass{article}
\usepackage{amssymb}
\usepackage{amsfonts}
\usepackage{amsmath}

\newtheorem{theorem}{Theorem}[section]

\newtheorem{definition}[theorem]{Definition}

\newtheorem{lemma}[theorem]{Lemma}
\newtheorem{notation}[theorem]{Notation}

\newtheorem{proposition}[theorem]{Proposition}
\newtheorem{remark}[theorem]{Remark}

\newenvironment{proof}[1][Proof]{\noindent\textbf{#1.} }{\ \rule{0.5em}{0.5em}}
\input{tcilatex}
\begin{document}

\title{A remainder estimate for branched rough differential equations}
\author{Danyu Yang\thanks{%
danyuyang@cqu.edu.cn}}
\date{\vspace{-3ex}}
\maketitle

\begin{abstract}
Based on two isomorphisms of Hopf algebras, we provide a bound in the
optimal order on the remainder of the truncated Taylor expansion for
controlled differential equations driven by branched rough paths.
\end{abstract}

\section{Introduction}

In the seminal paper \cite{lyons1998differential}, Lyons builds the theory
of rough paths. The theory solves rough differential equations (RDEs) of the
form%
\begin{equation*}
dy_{t}=f\left( y_{t}\right) dx_{t},y_{0}=\xi ,
\end{equation*}%
where $x$ can be highly oscillating.\ Under a Lipschitz condition on the
vector field, Lyons proves the unique solvability of the differential
equation, and the solution obtained is continuous with respect to the
driving signal in rough paths metric. The theory has an embedded component
in stochastic analysis, and $x$ can be Brownian motion, continuous
semi-martingales, Markov processes, Gaussian processes \cite%
{friz2010multidimensional} etc.

In 1972, Butcher \cite{butcher1972algebraic} identifies a group structure in
a class of integration methods including Runge-Kutta methods and Picard
iterations, where each method can be represented by a family of real-valued
functions indexed by rooted trees. In\ \cite{grossman1989hopf}, Grossman and
Larson describe several Hopf algebras associated with families of trees. One
Hopf algebra of simple rooted trees, with product \cite[(3.1)]%
{grossman1989hopf} and coproduct \cite[p.199]{grossman1989hopf}, is
particularly relevant to our setting, which we refer to as the Grossman
Larson Hopf algebra, denoted as $\mathcal{H}$. In \cite{connes1998hopf},
Connes and Kreimer describe a Hopf algebra based on rooted trees \cite[%
Section 2]{connes1998hopf} to disentangle the intricate combinatorics of
divergences in quantum field theory. We call this Hopf algebra the Connes
Kreimer Hopf algebra, denoted as $\mathcal{H}_{R}$. The group identified by
Butcher is the group of characters of $\mathcal{H}_{R}$ \cite%
{connes1999lessons}. Based on \cite{panaite2000relating,
foissy2002algebresb, hoffman2003combinatorics}, $\mathcal{H}$ is isomorphic
to the graded dual of $\mathcal{H}_{R}$.

Rough differential equations are originally driven by geometric rough paths
over Banach spaces \cite{lyons1998differential}. Geometric rough paths
satisfy an abstract integration by parts formula, and take values in a
nilpotent Lie group. The nilpotent Lie group can be expressed as a truncated
group of characters of the shuffle Hopf algebra \cite{reutenauer1993free}.
In \cite{gubinelli2010ramification}, Gubinelli defines branched rough paths.
Branched rough paths take values in a truncated group of characters of a
labeled Connes Kreimer Hopf algebra. Both geometric and branched rough paths
are of finite $p$-variation in rough paths metric, and encode information
needed to construct solutions to differential equations. There exists a Hopf
algebra homomorphism from the Connes Kreimer Hopf algebra onto the shuffle
Hopf algebra, which induces an embedding of geometric rough paths into
branched rough paths. On the other hand, the Grossman Larson algebra is
freely generated by a collection of trees \cite{foissy2002finite,
chapoton2010free}. Based on the freeness of Grossman Larson algebra,
Boedihardjo and Chevyrev construct an isomorphism between branched rough
paths and a class of geometric rough paths \cite{boedihardjo2019isomorphism}%
. As a result, a branched RDE can be expressed as a geometric RDE driven by
a $\Pi $-rough path defined by Gyurk\'{o} \cite{gyurko2016differential}.

Based on the isomorphism between $\mathcal{H}$ and the graded dual of $%
\mathcal{H}_{R}$ \cite{panaite2000relating, foissy2002algebresb,
hoffman2003combinatorics}, we clarify a relationship between rough paths
taking values in the truncated group of characters of $\mathcal{H}_{R}$ and
rough paths taking values in the truncated group of grouplike elements in $%
\mathcal{H}$ (Proposition \ref{Proposition transfer between rough paths in
two groups}). Based on this relationship and the freeness of the Grossman
Larson algebra, sub-Riemannian geometry \cite[Section 7.5]%
{friz2010multidimensional} and the neo-classical inequality \cite%
{lyons1998differential, hara2010fractional}, which are typical geometric
rough paths tools, can be applied to branched rough paths. As an
application, we provide an estimate for the remainder of the truncated
Taylor expansion for controlled differential equations driven by branched
rough paths (Theorem \ref{main theorem}). The remainder estimate is in the
optimal order (Remark \ref{Remark optimal bound}), which is pleasantly
surprising noting the rapid increase of the dimensions of simple rooted
trees.

\section{Notations and Results}

A rooted tree is a finite connected graph that has no cycle with a
distinguished vertex called root. We call a rooted tree a tree. We assume
trees are non-planar, which means that the children trees of each vertex are
commutative. A forest is a commutative monomial of trees. The degree $%
\left\vert \rho \right\vert $ of a forest $\rho $ is given by the number of
vertices. For a given label set, a labeled forest is a forest for which each
vertex is attached with a label.

Denote the label set $\mathcal{L}:=\left\{ 1,2,\dots ,d\right\} $. Let $F_{%
\mathcal{L}}$ ($T_{\mathcal{L}}$) denote the set of $\mathcal{L}$-labeled
forests (trees) of degree greater or equal to $1$. Let $F_{\mathcal{L}}^{N}$
($T_{\mathcal{L}}^{N}$) denote the set of elements in $F_{\mathcal{L}}$ ($T_{%
\mathcal{L}}$) of degree $1,\dots ,N$.

Let $G_{\mathcal{L}}^{N}$ denote the set of degree-$N$ characters of the $%
\mathcal{L}$-labeled Connes Kreimer Hopf algebra\ \cite[p.214]%
{connes1998hopf}. $a$ is an element of $G_{\mathcal{L}}^{N}$, if $a$ is an $%
\mathbb{R}
$-linear map $%
\mathbb{R}
F_{\mathcal{L}}^{N}\rightarrow 
\mathbb{R}
$ that satisfies%
\begin{equation*}
\left( a,\rho _{1}\right) \left( a,\rho _{2}\right) =\left( a,\rho _{1}\rho
_{2}\right)
\end{equation*}%
for every $\rho _{1},\rho _{2}\in F_{\mathcal{L}}^{N},\left\vert \rho
_{1}\right\vert +\left\vert \rho _{2}\right\vert \leq N$, where $\rho
_{1}\rho _{2}$ denotes the multiplication of commutative monomials of
trees.\ Let\ $\triangle $\ denote the coproduct of the Connes Kreimer Hopf
algebra based on admissible cuts \cite[p.215]{connes1998hopf}. Then $G_{%
\mathcal{L}}^{N}$ is a group with the multiplication given by%
\begin{equation*}
\left( ab,\rho \right) :=\left( a\otimes b,\triangle \rho \right)
\end{equation*}%
for every $\rho \in F_{\mathcal{L}}^{N}$.\ $G_{\mathcal{L}}^{N}$ is a
labeled\ truncated Butcher group\ \cite{butcher1972algebraic}. We equip $G_{%
\mathcal{L}}^{N}\ $with the norm%
\begin{equation}
\left\Vert a\right\Vert :=\max_{\rho \in F_{\mathcal{L}}^{N}}\left\vert
\left( a,\rho \right) \right\vert ^{\frac{1}{\left\vert \rho \right\vert }}.
\label{norm on GLN}
\end{equation}

With $\mathcal{L}=\left\{ 1,2,\dots ,d\right\} $, let $\mathcal{H}_{\mathcal{%
L}}$ denote the $\mathcal{L}$-labeled Grossman Larson Hopf algebra with
product\ \cite[(3.1)]{grossman1989hopf} and coproduct \cite[p.199]%
{grossman1989hopf}. Denote the product and coproduct of $\mathcal{H}_{%
\mathcal{L}}$ as $\ast $ and $\delta $ respectively. We consider $\mathcal{H}%
_{\mathcal{L}}$ as a Hopf algebra of labeled forests (by deleting the
additional root in \cite{grossman1989hopf}). An element $a\in \mathcal{H}_{%
\mathcal{L}}$ is \textit{grouplike} if $\delta a=a\otimes a$. Let $\mathcal{G%
}_{\mathcal{L}}$ denote the group of grouplike elements in $\mathcal{H}_{%
\mathcal{L}}$. For integer $N\geq 1$, the set of series $b=\sum_{\rho \in F_{%
\mathcal{L}},\left\vert \rho \right\vert >N}\left( b,\rho \right) \rho $
form an ideal of $\mathcal{H}_{\mathcal{L}}$. Let $\mathcal{H}_{\mathcal{L}%
}^{N}$ denote the quotient algebra. Denote $\mathcal{G}_{\mathcal{L}}^{N}:=%
\mathcal{G}_{\mathcal{L}}\cap \mathcal{H}_{\mathcal{L}}^{N}$. $\mathcal{G}_{%
\mathcal{L}}^{N}$ is a group. We equip $\mathcal{G}_{\mathcal{L}}^{N}$ with
a continuous homogeneous norm.

Let $\bullet _{a}$ denote the labeled tree of one vertex with a label $a\in 
\mathcal{L}$ on the vertex. Let $\left[ \tau _{1}\cdots \tau _{k}\right]
_{a} $ denote the labeled tree with children trees $\tau _{1},\dots ,\tau
_{k}$ on the root and a label $a\in \mathcal{L}$ on the root. Define $\sigma
:F_{\mathcal{L}}\rightarrow 
\mathbb{N}
$ as the symmetry factor given inductively by $\sigma \left( \bullet
_{a}\right) :=1$ and 
\begin{equation*}
\sigma \left( \tau _{1}^{n_{1}}\cdots \tau _{k}^{n_{k}}\right) =\sigma
\left( \left[ \tau _{1}^{n_{1}}\cdots \tau _{k}^{n_{k}}\right] _{a}\right)
:=n_{1}!\cdots n_{k}!\sigma \left( \tau _{1}\right) ^{n_{1}}\cdots \sigma
\left( \tau _{k}\right) ^{n_{k}},
\end{equation*}%
where $\tau _{i}\in T_{\mathcal{L}}$ are different labeled trees (with
labels counted). $\sigma $ is the order of the permutation group on vertices
in a tree that keeps the tree unchanged.

Let $\bigtriangleup $ denote the coproduct of the Connes Kreimer Hopf
algebra, and let $\ast $ denote the product of the Grossman Larson Hopf
algebra. Based on \cite[Theorem 43]{foissy2002algebresb} and \cite[%
Proposition 4.4]{hoffman2003combinatorics}, for $\rho \in F_{\mathcal{L}}$,%
\begin{equation}
\bigtriangleup \rho =\sum_{\rho _{i}\in F_{\mathcal{L}}}\frac{\sigma \left(
\rho \right) }{\sigma \left( \rho _{1}\right) \sigma \left( \rho _{2}\right) 
}\left( \rho _{1}\ast \rho _{2},\rho \right) \rho _{1}\otimes \rho _{2}.
\label{inner GL product and CK coproduct}
\end{equation}

\begin{definition}
Suppose $G$ is a group with norm $\left\Vert \cdot \right\Vert $. Let $X:%
\left[ 0,T\right] \rightarrow \left( G,\left\Vert \cdot \right\Vert \right) $%
. Denote%
\begin{equation*}
X_{s,t}:=X_{s}^{-1}X_{t}.
\end{equation*}%
For $p\geq 1$, define%
\begin{equation*}
\left\Vert X\right\Vert _{p-var,\left[ 0,T\right] }:=\left(
\sup_{0=t_{0}<\cdots <t_{n}=T,n\geq 1}\sum_{i=0}^{n-1}\left\Vert
X_{t_{i},t_{i+1}}\right\Vert ^{p}\right) ^{\frac{1}{p}}.
\end{equation*}
\end{definition}

For $p\geq 1$, let $\left[ p\right] $ denote the largest integer that is
less or equal to $p$.

\begin{definition}
For $p\geq 1$, $X$\ is a branched $p$-rough path if $X:\left[ 0,T\right]
\rightarrow G_{\mathcal{L}}^{\left[ p\right] }$ is continuous and of finite $%
p$-variation.
\end{definition}

\begin{proposition}
\label{Proposition transfer between rough paths in two groups}For $p\geq 1$,
suppose $X:\left[ 0,T\right] \rightarrow G_{\mathcal{L}}^{\left[ p\right] }$
is a branched $p$-rough path. Define $\bar{X}:\left[ 0,T\right] \rightarrow
\left( F_{\mathcal{L}}^{\left[ p\right] }\rightarrow 
\mathbb{R}
\right) $ as%
\begin{equation*}
\left( \bar{X}_{t},\rho \right) :=\frac{\left( X_{t},\rho \right) }{\sigma
\left( \rho \right) }
\end{equation*}%
for $t\in \left[ 0,T\right] $ and $\rho \in F_{\mathcal{L}}^{\left[ p\right]
}$. Then $\bar{X}$\ takes values in $\mathcal{G}_{\mathcal{L}}^{\left[ p%
\right] }$, is continuous and of finite $p$-variation, and 
\begin{equation}
\left( \bar{X}_{s,t},\rho \right) =\frac{\left( X_{s,t},\rho \right) }{%
\sigma \left( \rho \right) }  \label{relation between X and Xbar 2}
\end{equation}%
for $0\leq s\leq t\leq T$ and $\rho \in F_{\mathcal{L}}^{\left[ p\right] }$.
For integer $N\geq \left[ p\right] +1$, there exists a unique extension of $%
X $ resp. $\bar{X}$ to a continuous path of finite $p$-variation taking
values in $G_{\mathcal{L}}^{N}$ resp. $\mathcal{G}_{\mathcal{L}}^{N}$. Still
denote their extension as $X$ resp. $\bar{X}$. Then $\left( \ref{relation
between X and Xbar 2}\right) $ holds for $0\leq s\leq t\leq T$ and $\rho \in
F_{\mathcal{L}}^{N}$.
\end{proposition}

\begin{remark}
Since the Grossman Larson algebra is free on a collection of trees \cite%
{foissy2002finite, chapoton2010free}, $\bar{X}$ acts as a bridge between $X$
and geometric rough paths. In particular, sub-Riemannian geometry technique 
\cite[Section 7.5]{friz2010multidimensional} and the neo-classical
inequality \cite{lyons1998differential, hara2010fractional} can be applied
to $\bar{X}$. Then results are transferred back to $X$ based on $\left( \ref%
{relation between X and Xbar 2}\right) $.
\end{remark}

Let $L\left( 
\mathbb{R}
^{d},%
\mathbb{R}
^{e}\right) $ denote the set of linear mappings from $%
\mathbb{R}
^{d}$ to $%
\mathbb{R}
^{e}$. For $f=\left( f_{1},\dots ,f_{d}\right) :%
\mathbb{R}
^{e}\rightarrow L\left( 
\mathbb{R}
^{d},%
\mathbb{R}
^{e}\right) $ that is sufficiently smooth, define $f:T_{\mathcal{L}%
}\rightarrow \left( 
\mathbb{R}
^{e}\rightarrow 
\mathbb{R}
^{e}\right) $ inductively as 
\begin{equation}
f\left( \bullet _{a}\right) :=f_{a}\text{ and }f\left( \left[ \tau
_{1}\cdots \tau _{k}\right] _{a}\right) :=\left( d^{k}f_{a}\right) \left(
f\left( \tau _{1}\right) \cdots f\left( \tau _{k}\right) \right)
\label{Definition of f(tau)}
\end{equation}%
for $\tau _{i}\in T_{\mathcal{L}}$ and $a\in \mathcal{L}$, where $d^{k}f_{a}$
denotes the $k$th Fr\'{e}chet derivative of $f_{a}$.

Lipschitz functions and norms are defined as in\ \cite[Definition 1.2.4,
p.230]{lyons1998differential}. For $\gamma >1$, let $\lfloor \gamma \rfloor $
denote the largest integer that is strictly less than $\gamma $.

\begin{theorem}
\label{main theorem}For $\gamma >p\geq 1$, suppose $X:\left[ 0,T\right]
\rightarrow G_{\mathcal{L}}^{\left[ p\right] }\ $is a branched $p$-rough
path over base space $%
\mathbb{R}
^{d}$, and $f:%
\mathbb{R}
^{e}\rightarrow L\left( 
\mathbb{R}
^{d},%
\mathbb{R}
^{e}\right) $ is $\limfunc{Lip}\left( \gamma \right) $. Let $y$ denote the
unique solution of the branched rough differential equation%
\begin{equation*}
dy_{t}=f\left( y_{t}\right) dX_{t},y_{0}=\xi \in 
\mathbb{R}
^{e}.
\end{equation*}%
Then with $N:=\lfloor \gamma \rfloor $, there exist two positive constants $%
c_{p,d,\omega \left( 0,T\right) }^{1}\ $and $c_{p,d}^{2}\ $such that,%
\begin{equation}
\left\Vert y_{t}-y_{s}-\sum_{\tau \in T_{\mathcal{L}}^{N}}f\left( \tau
\right) \left( y_{s}\right) \frac{\left( X_{s,t},\tau \right) }{\sigma
\left( \tau \right) }\right\Vert \leq c_{p,d,\omega \left( 0,T\right) }^{1}N!%
\frac{\omega \left( s,t\right) ^{\frac{N+1}{p}}}{\left( \frac{N+1}{p}\right)
!}  \label{main estimate}
\end{equation}%
where $\omega \left( s,t\right) :=c_{p,d}^{2}\left\Vert f\right\Vert _{%
\limfunc{Lip}\left( \gamma \right) }^{p}\left\Vert X\right\Vert _{p-var,%
\left[ s,t\right] }^{p}$.
\end{theorem}

The solution to branched RDEs is defined as in \cite[Section 8.1]%
{gubinelli2010ramification}. The existing Taylor remainder estimates for the
solution of branched RDEs only deal with the case $N=\left[ p\right] $ \cite[%
Theorem 8.8]{gubinelli2010ramification}. Theorem \ref{main theorem}
considers the general case $N\geq \left[ p\right] $, and the estimate $%
\left( \ref{main estimate}\right) $ is in the optimal order (Remark \ref%
{Remark optimal bound}).

\begin{remark}
When $p=1$ and $\left\Vert f\right\Vert _{\limfunc{Lip}\left( \infty \right)
}<\infty $, suppose $\left\Vert f\right\Vert _{\limfunc{Lip}\left( \infty
\right) }\left\Vert X\right\Vert _{1-var,\left[ s,t\right] }<\left(
c_{1,d}^{2}\right) ^{-1}$. Then the Taylor series converges \cite[Theorem 5.1%
]{gubinelli2010ramification}.
\end{remark}

Suppose $x:\left[ 0,T\right] \rightarrow 
\mathbb{R}
^{d}$ is continuous and of bounded variation, and $f:%
\mathbb{R}
^{e}\rightarrow L\left( 
\mathbb{R}
^{d},%
\mathbb{R}
^{e}\right) $ is sufficiently smooth. Consider the ODE%
\begin{equation*}
dy_{t}=f\left( y_{t}\right) dx_{t},y_{0}=\xi .
\end{equation*}%
Based on the fundamental theorem of calculus, for $s\leq t$,%
\begin{equation*}
y_{t}=y_{s}+\sum_{k=1}^{N}f^{\circ k}\left( y_{s}\right)
X_{s,t}^{k}+\idotsint\limits_{s<u_{1}<\cdots <u_{N+1}<t}f^{\circ \left(
N+1\right) }\left( y_{u_{1}}\right) dx_{u_{1}}\cdots dx_{u_{N+1}}
\end{equation*}%
where $f^{\circ 1}:=f,f^{\circ \left( k+1\right) }:=df^{\circ k}\left(
f\right) $ and $X_{s,t}^{k}:=\idotsint\limits_{s<u_{1}<\cdots
<u_{k}<t}dx_{u_{1}}\otimes \cdots \otimes dx_{u_{k}}$.

\begin{remark}
\label{Remark optimal bound}Suppose $X$ is a geometric $p$-rough path. Let $%
X_{s,t}^{N+1}$ denote the $\left( N+1\right) $th level element of $X$ on $%
\left[ s,t\right] $. Based on Lyons' extension Theorem \cite[Theorem 2.2.1,
p.242]{lyons1998differential}, there exists a positive constant $\beta _{p}$
such that 
\begin{equation*}
\left\Vert X_{s,t}^{N+1}\right\Vert \leq \frac{\omega \left( s,t\right) ^{%
\frac{N+1}{p}}}{\beta _{p}\left( \frac{N+1}{p}\right) !}
\end{equation*}%
for every $s\leq t$ and every $N$. On the other hand, consider $f\left(
t\right) =e^{-t}$. Then $\left\Vert f\right\Vert _{Lip\left( n\right) }=1$
on $t\geq 0$ for $n=1,2,\dots $ and 
\begin{equation*}
\left\vert f^{\circ \left( N+1\right) }\left( 0\right) \right\vert =N!
\end{equation*}%
The estimate $\left( \ref{main estimate}\right) $ states that the remainder
can be bounded similarly to $f^{\circ \left( N+1\right) }\left( y_{s}\right)
X_{s,t}^{N+1}$ that is in the optimal order even in the geometric case. The
dimension of trees contributes a geometric increase factor\footnote{%
Based on \cite[Section I.5.2]{flajolet2009analytic}, the number of unlabeled
simple rooted trees is EIS A000081, and $H_{n}\sim \lambda \beta
^{n}n^{-3/2} $ with $\lambda $ approximately $0.43992$ and $\beta $
approximately $2.95576 $.} that is part of the control $\omega $.
\end{remark}

The proof of Theorem \ref{main theorem} is based on a mathematical induction
that is an inhomogeneous analogue of \cite{boedihardjo2015uniform}. The main
estimate $\left( \ref{main estimate}\right) $ is obtained by exploring the
sub-Riemannian geometry of the truncated group of grouplike elements in the
Grossman Larson Hopf algebra. The sub-Riemannian geometry structure is
similar to that of the nilpotent Lie group \cite[Theorem 7.32]%
{friz2010multidimensional}. The factor $\left( \left( N+1\right) /p\right) !$
is obtained by the neo-classical inequality \cite{lyons1998differential,
hara2010fractional}. The tree neo-classical inequality is known to be false 
\cite[Section 3]{boedihardjo2018decay}. Since the Grossman Larson algebra is
free on a collection of trees, the analysis can be transferred back to the
Tensor algebra where the neo-classical inequality holds. Our estimates rely
critically on the simple fact that the number of words generated by a finite
set of letters grows geometrically (Lemma \ref{Lemma bound on Tn}).

\section{Proofs}

\begin{proof}[Proof of Proposition \protect\ref{Proposition transfer between
rough paths in two groups}]
Since $\left( X_{t},\rho \right) /\sigma \left( \rho \right) =\left( \bar{X}%
_{t},\rho \right) $ for $\rho \in F_{\mathcal{L}}$, $\left\vert \rho
\right\vert =1,\dots ,\left[ p\right] $, it can be proved inductively based
on $\left( \ref{inner GL product and CK coproduct}\right) $ that for\ $\rho
\in F_{\mathcal{L}}$, $\left\vert \rho \right\vert =1,\dots ,\left[ p\right] 
$, and $s\leq t$,%
\begin{equation*}
\left( \bar{X}_{s,t},\rho \right) =\frac{\left( X_{s,t},\rho \right) }{%
\sigma \left( \rho \right) }.
\end{equation*}

The existence and uniqueness of the extension of $X$ and $\bar{X}$ can be
proved similarly to \cite[Theorem 2.2.1]{lyons1998differential}. Based on $%
\left( \ref{inner GL product and CK coproduct}\right) $, when $\rho \in F_{%
\mathcal{L}}$, $\left\vert \rho \right\vert =n,n\geq \left[ p\right] +1,$%
\begin{eqnarray*}
&&\left( X_{s,t},\rho \right) \\
&=&\sigma \left( \rho \right) \lim_{D\subset \left[ s,t\right] ,\left\vert
D\right\vert \rightarrow 0}\sum_{\rho _{i}\in F_{\mathcal{L}},\left\vert
\rho _{i}\right\vert <\left\vert \rho \right\vert }\frac{\left(
X_{t_{0},t_{1}},\rho _{1}\right) }{\sigma \left( \rho _{1}\right) }\cdots 
\frac{\left( X_{t_{k-1},t_{k}},\rho _{k}\right) }{\sigma \left( \rho
_{k}\right) }\left( \rho _{1}\ast \cdots \ast \rho _{k},\rho \right) \\
&=&\sigma \left( \rho \right) \lim_{D\subset \left[ s,t\right] ,\left\vert
D\right\vert \rightarrow 0}\sum_{\rho _{i}\in F_{\mathcal{L}},\left\vert
\rho _{i}\right\vert <\left\vert \rho \right\vert }\left( \bar{X}%
_{t_{0},t_{1}},\rho _{1}\right) \cdots \left( \bar{X}_{t_{k-1},t_{k}},\rho
_{k}\right) \left( \rho _{1}\ast \cdots \ast \rho _{k},\rho \right) \\
&=&\sigma \left( \rho \right) \left( \bar{X}_{s,t},\rho \right) .
\end{eqnarray*}
\end{proof}

Based on \cite[Section 8]{foissy2002finite} and \cite{chapoton2010free}, the
Grossman Larson algebra is freely generated by a collection of unlabeled
trees. Denote this collection of trees as $\mathcal{B}$. Let $\mathcal{B}_{%
\mathcal{L}}$ denote the $\mathcal{L}$-labeled version of $\mathcal{B}$ with 
$\mathcal{L}=\left\{ 1,2,\dots ,d\right\} $.

\begin{notation}
\label{Notation tau1 tau2 ... tauK}Let $\mathcal{B}_{\mathcal{L}}^{\left[ p%
\right] }=\left\{ \upsilon _{1},\dots ,\upsilon _{K}\right\} $ denote the
set of elements in $\mathcal{B}_{\mathcal{L}}$ with degree less or equal to $%
\left[ p\right] $.
\end{notation}

\begin{definition}
\label{Definition geodesic norm}For $a\in \mathcal{G}_{\mathcal{L}}^{\left[ p%
\right] }$, define%
\begin{equation}
\left\Vert a\right\Vert ^{\prime }:=\inf_{x}\sum_{i=1}^{K}\left\Vert
x^{\upsilon _{i}}\right\Vert _{1-var}^{\frac{1}{\left\vert \upsilon
_{i}\right\vert }}  \label{norm dash}
\end{equation}%
where the infimum is taken over all continuous bounded variation paths $%
x=\left( x^{\upsilon _{1}},\dots ,x^{\upsilon _{K}}\right) :\left[ 0,1\right]
\rightarrow 
\mathbb{R}
^{K}$ that satisfy%
\begin{equation*}
\left( a,\upsilon _{i_{1}}\ast \cdots \ast \upsilon _{i_{k}}\right)
=\idotsint\limits_{0<u_{1}<\cdots <u_{k}<1}dx_{u_{1}}^{\upsilon
_{i_{1}}}\cdots dx_{u_{k}}^{\upsilon _{i_{k}}}
\end{equation*}%
for $\upsilon _{i_{j}}\in \mathcal{B}_{\mathcal{L}}^{\left[ p\right] }$, $%
\left\vert \upsilon _{i_{1}}\right\vert +\cdots +\left\vert \upsilon
_{i_{k}}\right\vert \leq \left[ p\right] $. The infimum in $\left( \ref{norm
dash}\right) $ can be obtained at a continuous bounded variation path $x$,
which is called a \emph{geodesic} associated with $a\in \mathcal{G}_{%
\mathcal{L}}^{\left[ p\right] }$.
\end{definition}

\begin{remark}
Such $x$ exists based on Chow-Rashevskii Theorem \cite[Theorem 7.28]%
{friz2010multidimensional}. Based on Arzel\`{a}-Ascoli Theorem and lower
semi-continuity of $1$-variation, the infimum can be obtained at some $x$
that is continuous and of bounded variation.
\end{remark}

For $a:F_{\mathcal{L}}^{N}\rightarrow 
\mathbb{R}
$ and $c>0$, define $\delta _{c}a:F_{\mathcal{L}}^{N}\rightarrow 
\mathbb{R}
$ as $\left( \delta _{c}a,\rho \right) :=c^{\left\vert \rho \right\vert
}\left( a,\rho \right) $. A norm $\left\Vert \cdot \right\Vert $ is
homogeneous if $\left\Vert \delta _{c}a\right\Vert =c\left\Vert a\right\Vert 
$ for every $c>0$ and every $a$. $\left\Vert \cdot \right\Vert ^{\prime }$
is a continuous homogeneous norm. The continuity of $\left\Vert \cdot
\right\Vert ^{\prime }$ can be proved similarly as \cite[Proposition 7.40(v)]%
{friz2010multidimensional}.

\begin{proposition}
\label{Proposition equivalency of norms}Continuous homogeneous norms on $%
\mathcal{G}_{\mathcal{L}}^{\left[ p\right] }$ are equivalent up to a
constant depending on $p$ and $d$.
\end{proposition}

\begin{proof}
The proof is similar to \cite[Theorem 7.44]{friz2010multidimensional}.
\end{proof}

\begin{lemma}
\label{Lemma bound 1-var of geodesic by omega}Let $x=\left( x^{\upsilon
_{1}},\dots ,x^{\upsilon _{K}}\right) :\left[ 0,1\right] \rightarrow 
\mathbb{R}
^{K}$ be a geodesic associated with $\bar{X}_{s,t}$. Then there exists $%
M_{p,d}>0$ such that%
\begin{equation*}
\left\Vert x^{\upsilon _{i}}\right\Vert _{1-var}\leq \left( M_{p,d}\right)
^{\left\vert \upsilon _{i}\right\vert }\left\Vert X\right\Vert _{p-var,\left[
s,t\right] }^{\left\vert \upsilon _{i}\right\vert }
\end{equation*}%
for every $\upsilon _{i}\in \mathcal{B}_{\mathcal{L}}^{\left[ p\right] }$.
\end{lemma}

\begin{proof}
Define a norm on $\mathcal{G}_{\mathcal{L}}^{\left[ p\right] }$ as $%
\left\Vert a\right\Vert _{1}:=\max_{\rho \in F_{\mathcal{L}}^{\left[ p\right]
}}\left\vert \left( a,\rho \right) \right\vert ^{\frac{1}{\left\vert \rho
\right\vert }}$. Based on the definition of $\left\Vert \cdot \right\Vert
^{\prime }$, equivalency of continuous homogeneous norms as in Proposition %
\ref{Proposition equivalency of norms} and that $\left( \bar{X}_{s,t},\rho
\right) =\left( X_{s,t},\rho \right) /\sigma \left( \rho \right) $, the
proposed inequality holds.
\end{proof}

\begin{notation}
Let $\mathcal{W}$ denote the set of finite sequences $t_{1}\cdots t_{k}$ of $%
t_{i}\in \mathcal{B}_{\mathcal{L}}^{\left[ p\right] }$, including the empty
sequence denoted as $\eta $. The degree $\left\vert w\right\vert $ of $%
w=t_{1}\cdots t_{k}$ is $\left\vert t_{1}\right\vert +\cdots +\left\vert
t_{k}\right\vert $. The degree of $\eta $ is $0$.
\end{notation}

\begin{lemma}
\label{Lemma bound on Tn}Let $T_{n}$ denote the number of elements in $%
\mathcal{W}$ of degree $n$. Then there exists $K_{p}\geq 1$ such that for $%
n=1,2,3,\dots $ 
\begin{equation*}
T_{n}\leq \left( K_{p}d\right) ^{n}.
\end{equation*}
\end{lemma}

\begin{proof}
Recall that $\mathcal{B}$ denotes the collection of trees that freely
generate the Grossman Larson algebra. For $i=1,2,\dots ,\left[ p\right] $,
let $l_{i}\ $denote the number of trees in $\mathcal{B}$ of degree $i$. Then 
$T_{n}\leq \sum_{i=1}^{\left[ p\right] }T_{n-i}l_{i}d^{i}$. Set $T_{0}=1$
and $T_{-n}=0$ for $n=1,2,\dots \left[ p\right] $. For $p\geq 1$, let $%
K_{p}\geq 1$ be a number such that $\sum_{i=1}^{\left[ p\right] }l_{i}\left(
K_{p}\right) ^{-i}\leq 1$. Then it can be proved inductively $T_{n}\leq
\left( K_{p}d\right) ^{n}$.
\end{proof}

Define $I\left( x\right) :=x$ for $x\in 
\mathbb{R}
^{e}$.

\begin{notation}
For $t\in \mathcal{B}_{\mathcal{L}}^{\left[ p\right] }$ and $w\in \mathcal{W}
$, denote $F^{\eta }:=I$, $F^{t}:=f\left( t\right) $ as at $\left( \ref%
{Definition of f(tau)}\right) $ and 
\begin{equation*}
F^{tw}:=dF^{w}\left( f\left( t\right)\right).
\end{equation*}
\end{notation}

\begin{notation}
With $f\left( t_{i}\right) $ defined at $\left( \ref{Definition of f(tau)}%
\right) $, let $\psi _{f}$ denote the $%
\mathbb{R}
$-linear map from $%
\mathbb{R}
F_{\mathcal{L}}$ to differential operators, given by $\psi _{f}\left(
t_{1}\cdots t_{k}\right) \left( \varphi \right) :=d^{k}\varphi \left(
f\left( t_{1}\right) \cdots f\left( t_{k}\right) \right) $ for $t_{i}\in T_{%
\mathcal{L}}$ and smooth $\varphi :%
\mathbb{R}
^{e}\rightarrow 
\mathbb{R}
^{e}$.
\end{notation}

For trees $t_{i}$ and a forest $\rho $, define $\left( t_{1}\cdots
t_{k}\right) \curvearrowright \rho $ as the sum of $\left\vert \rho
\right\vert ^{k}$ forests that are obtained by linking each of the roots of $%
t_{i},i=1,\dots ,k$ to a vertex of $\rho $ by a new edge. Recall that $\ast $
denotes the product in the Grossman Larson Hopf algebra (we delete the
additional root in \cite{grossman1989hopf}). Then for trees $t$ and $t_{i}$, 
$t\ast \left( t_{1}\cdots t_{k}\right) =tt_{1}\cdots t_{k}+t\curvearrowright
\left( t_{1}\cdots t_{k}\right) $.

\begin{lemma}
\label{Lemma expression of Ft1...tk}With $f\left( t\right) $ defined at $%
\left( \ref{Definition of f(tau)}\right) $, for $t_{i}\in T_{\mathcal{L}%
},i=1,\dots ,k$,%
\begin{equation*}
F^{t_{1}\cdots t_{k}}=f\left( t_{1}\curvearrowright \left(
t_{2}\curvearrowright \cdots \left( t_{k-1}\curvearrowright t_{k}\right)
\right) \right) =\psi _{f}\left( t_{1}\ast \cdots \ast t_{k}\right) \left(
I\right) .
\end{equation*}
\end{lemma}

\begin{proof}
Since $df\left( t_{2}\right) f\left( t_{1}\right) =f\left(
t_{1}\curvearrowright t_{2}\right) $ for $t_{1},t_{2}\in T_{\mathcal{L}}$,
the first equality holds. For trees $t_{1},t_{2}$ and a forest $\rho $, $%
t_{1}\curvearrowright \left( \rho \curvearrowright t_{2}\right) =\left(
t_{1}\ast \rho \right) \curvearrowright t_{2}$. Then it can be proved
inductively that, for $t_{i}\in T_{\mathcal{L}}$, $t_{1}\curvearrowright
\left( t_{2}\curvearrowright \cdots \left( t_{k-1}\curvearrowright
t_{k}\right) \right) =\left( t_{1}\ast t_{2}\ast \cdots \ast t_{k-1}\right)
\curvearrowright t_{k}$. Then the second equality holds based on 
\begin{equation*}
f\left( \left( t_{1}\ast t_{2}\ast \cdots \ast t_{k-1}\right)
\curvearrowright t_{k}\right) =\psi _{f}\left( t_{1}\ast t_{2}\ast \cdots
\ast t_{k-1}\ast t_{k}\right) \left( I\right) .
\end{equation*}
\end{proof}

For $\gamma >p\geq 1$, suppose $X:\left[ 0,T\right] \rightarrow G_{\mathcal{L%
}}^{\left[ p\right] }$ is a branched $p$-rough path and suppose $f:%
\mathbb{R}
^{e}\rightarrow L\left( 
\mathbb{R}
^{d},%
\mathbb{R}
^{e}\right) $ is $\limfunc{Lip}\left( \gamma \right) $. Define $\omega
:\left\{ \left( s,t\right) |0\leq s\leq t\leq T\right\} \rightarrow \lbrack
0,\infty )$ as 
\begin{equation*}
\omega \left( s,t\right) :=\left\Vert f\right\Vert _{\limfunc{Lip}\left(
\gamma \right) }^{p}\left\Vert X\right\Vert _{p-var,\left[ s,t\right] }^{p}.
\end{equation*}%
By rescaling $\left\Vert f\right\Vert _{\limfunc{Lip}\left( \gamma \right)
}^{-1}f$ and $\delta _{\left\Vert f\right\Vert _{\limfunc{Lip}\left( \gamma
\right) }}X$, we assume $\left\Vert f\right\Vert _{\limfunc{Lip}\left(
\gamma \right) }=1$.

Denote $N:=\lfloor \gamma \rfloor $ and $\left\{ \gamma \right\} :=\gamma
-\lfloor \gamma \rfloor $.

\begin{lemma}
\label{Lemma for bound of composition of vector fields}For $w\in \mathcal{W}%
,\left\vert w\right\vert \leq N$, 
\begin{equation*}
\left\Vert F^{w}\left( y_{1}\right) -F^{w}\left( y_{2}\right) \right\Vert
\leq \left\vert w\right\vert !\left\Vert y_{1}-y_{2}\right\Vert
\end{equation*}
for $y_{i}\in 
\mathbb{R}
^{e}$. For $w\in \mathcal{W},\left\vert w\right\vert <N$ and $t\in $ $%
\mathcal{B}_{\mathcal{L}}^{\left[ p\right] }$, 
\begin{equation*}
\sup_{y\in 
\mathbb{R}
^{e}}\left\Vert F^{tw}\left( y\right) \right\Vert \leq \left( N-1\right) !
\end{equation*}
\end{lemma}

\begin{proof}
All trees here are labeled by $\mathcal{L}=\left\{ 1,2,\dots ,d\right\} $.
Based on Lemma \ref{Lemma expression of Ft1...tk}, $F^{t_{1}\cdots
t_{k}}=f\left( t_{1}\curvearrowright \left( t_{2}\curvearrowright \cdots
\left( t_{k-1}\curvearrowright t_{k}\right) \right) \right) $. Then $%
F^{t_{1}\cdots t_{k}}$ is the sum of the image of 
\begin{equation*}
\left\vert t_{k}\right\vert \left( \left\vert t_{k}\right\vert +\left\vert
t_{k-1}\right\vert \right) \cdots \left( \left\vert t_{k}\right\vert
+\left\vert t_{k-1}\right\vert +\cdots +\left\vert t_{2}\right\vert \right)
\end{equation*}%
trees. Hence, for $w\in \mathcal{W}$, the number of trees in $F^{w}$ is
bounded by $\left( \left\vert w\right\vert -1\right) !$. Each of these trees 
$t$ is of degree $\left\vert w\right\vert $ and corresponds to $f\left(
t\right) :%
\mathbb{R}
^{e}\rightarrow 
\mathbb{R}
^{e}$ that is at least $\limfunc{Lip}\left( 1+\left\{ \gamma \right\}
\right) $ as $\left\vert w\right\vert \leq N$. Then $df\left( t\right) $ is
a sum of $\left\vert w\right\vert $ terms, as the differential $d$ chooses a
vertex in $t$. Hence, $df\left( t\right) $ is bounded by $\left\vert
w\right\vert $, because $f$ and its derivatives of order up to $N$ are
uniformly bounded by $1$ (we rescaled $f$ by $\left\Vert f\right\Vert _{%
\limfunc{Lip}\left( \gamma \right) }^{-1}$). As a result, for each tree $t$
of degree $\left\vert w\right\vert $, $\left\Vert f\left( t\right) \left(
y_{1}\right) -f\left( t\right) \left( y_{2}\right) \right\Vert \leq
\left\Vert df\left( t\right) \right\Vert _{\infty }\left\Vert
y_{1}-y_{2}\right\Vert \leq \left\vert w\right\vert \left\Vert
y_{1}-y_{2}\right\Vert $. Then the first estimate follows, as there are at
most $\left( \left\vert w\right\vert -1\right) !$ such trees in $F^{w}$.

For $w\in \mathcal{W},\left\vert w\right\vert <N$ and $t\in $ $\mathcal{B}_{%
\mathcal{L}}^{\left[ p\right] }$, the number of trees in $F^{tw}$ is bounded
by $\left\vert w\right\vert !\leq \left( N-1\right) !$. Each tree
corresponds to a map that is bounded on $%
\mathbb{R}
^{e}$ by $1$.
\end{proof}

Recall that $\mathcal{B}_{\mathcal{L}}^{\left[ p\right] }=\left\{ \upsilon
_{1},\dots ,\upsilon _{K}\right\} $. For $s\leq t$, let $x=\left(
x^{\upsilon _{1}},\dots ,x^{\upsilon _{K}}\right) :\left[ s,t\right]
\rightarrow 
\mathbb{R}
^{K}$ be a geodesic associated with $\bar{X}_{s,t}$. With $f\left( \upsilon
_{i}\right) $ defined at $\left( \ref{Definition of f(tau)}\right) $, let $%
y^{s,t}$ denote the unique solution of the ODE 
\begin{equation*}
dy_{u}^{s,t}=\sum_{i=1}^{K}f\left( \upsilon _{i}\right) \left(
y_{u}^{s,t}\right) dx_{u}^{\upsilon _{i}},y_{s}^{s,t}=y_{s},
\end{equation*}%
where $y$ denotes the unique solution of the branched RDE 
\begin{equation*}
dy_{t}=f\left( y_{t}\right) dX_{t},y_{0}=\xi .
\end{equation*}%
The existence and uniqueness of $y$ is based on \cite[Theorem 22]%
{lyons2015theory}.

\begin{lemma}
\label{Lemma expression of Taylor expansion}For $w\in \mathcal{W}$, $%
\left\vert w\right\vert =N-\left[ p\right] ,\dots ,N-1$, 
\begin{eqnarray*}
&&F^{w}\left( y_{t}^{s,t}\right) -F^{w}\left( y_{s}\right) \\
&&-\dsum\limits_{\left\vert t_{1}\right\vert +\cdots +\left\vert
t_{k}\right\vert =1,\dots ,N-\left\vert w\right\vert }F^{t_{1}\cdots
t_{k}w}\left( y_{s}\right) \idotsint\limits_{s<u_{1}<\cdots
<u_{k}<t}dx_{u_{1}}^{t_{1}}\cdots dx_{u_{k}}^{t_{k}} \\
&=&\dsum\limits_{\left\vert t_{1}\right\vert +\cdots +\left\vert
t_{k}\right\vert =N-\left\vert w\right\vert }\quad
\idotsint\limits_{s<u_{1}<\cdots <u_{k}<t}\left( F^{t_{1}\cdots
t_{k}w}\left( y_{u_{1}}^{s,t}\right) -F^{t_{1}\cdots t_{k}w}\left(
y_{s}\right) \right) dx_{u_{1}}^{t_{1}}\cdots dx_{u_{k}}^{t_{k}} \\
&&+\sum_{\substack{ \left\vert t_{2}\right\vert +\cdots +\left\vert
t_{k}\right\vert <N-\left\vert w\right\vert  \\ \left\vert t_{1}\right\vert
+\left\vert t_{2}\right\vert +\cdots +\left\vert t_{k}\right\vert
>N-\left\vert w\right\vert }}\,\;\idotsint\limits_{s<u_{1}<\cdots
<u_{k}<t}F^{t_{1}\cdots t_{k}w}\left( y_{u_{1}}^{s,t}\right)
dx_{u_{1}}^{t_{1}}\cdots dx_{u_{k}}^{t_{k}}
\end{eqnarray*}%
where $t_{i},i=1,2,\dots \ $range over elements in $\mathcal{B}_{\mathcal{L}%
}^{\left[ p\right] }$.
\end{lemma}

\begin{proof}
The equality can be obtained by iteratively applying the fundamental theorem
of calculus.
\end{proof}

\begin{lemma}
\label{Lemma uniform bound on yst}%
\begin{equation*}
\sup_{u\in \left[ s,t\right] }\left\Vert y_{u}^{s,t}-y_{s}\right\Vert \leq
C\left( p,d,\omega \left( 0,T\right) \right) \omega \left( s,t\right) ^{%
\frac{1}{p}}
\end{equation*}
\end{lemma}

\begin{proof}
Since $K$ is the number of elements in $\mathcal{B}_{\mathcal{L}}^{\left[ p%
\right] }$ with $\mathcal{L}=\left\{ 1,\dots ,d\right\} $, $K$ only depends
on $p,d$. Since $\left\Vert f\right\Vert _{\limfunc{Lip}\left( \gamma
\right) }=1$, based on Lemma \ref{Lemma bound 1-var of geodesic by omega}, 
\begin{equation*}
\sup_{u\in \left[ s,t\right] }\left\Vert y_{u}^{s,t}-y_{s}\right\Vert \leq
\sum_{i=1}^{K}\left\Vert x^{\upsilon _{i}}\right\Vert _{1-var}\leq C\left(
p,d,\omega \left( 0,T\right) \right) \omega \left( s,t\right) ^{\frac{1}{p}}.
\end{equation*}
\end{proof}

Recall that $\ast $ denotes the product of the Grossman Larson Hopf algebra.

\begin{notation}
\label{Notation TX}Define $T^{X}\ $as%
\begin{equation*}
\left( T_{s,t}^{X},t_{1}\cdots t_{k}\right) :=\left( \bar{X}_{s,t},t_{1}\ast
\cdots \ast t_{k}\right)
\end{equation*}%
for $s\leq t$ and $t_{1}\cdots t_{k}\in \mathcal{W}$ for $t_{i}\in \mathcal{B%
}_{\mathcal{L}}^{\left[ p\right] },\left\vert t_{1}\right\vert +\cdots
+\left\vert t_{k}\right\vert \leq \left[ p\right] $.
\end{notation}

Denote $\beta _{p}:=p^{2}\left( 1+\sum_{n\geq 2}\left( \frac{2}{n}\right) ^{%
\frac{\left[ p\right] +1}{p}}\right) $.

\begin{lemma}
\label{Lemma initial estimate induction}Denote $\tilde{\omega}:=\left(
K_{p}d\right) ^{p}\omega $. For $w\in \mathcal{W}$, $\left\vert w\right\vert
=N-\left[ p\right] ,\dots ,N-1$,%
\begin{eqnarray*}
&&\left\Vert F^{w}\left( y_{t}^{s,t}\right) -F^{w}\left( y_{s}\right)
-\sum_{l\in \mathcal{W},\left\vert l\right\vert =1,\dots ,N-\left\vert
w\right\vert }F^{lw}\left( y_{s}\right) \left( T_{s,t}^{X},l\right)
\right\Vert \\
&\leq &C\left( p,d,\omega \left( 0,T\right) \right) N!\frac{\tilde{\omega}%
\left( s,t\right) ^{\frac{N+1-\left\vert w\right\vert }{p}}}{\beta
_{p}\left( \frac{N+1-\left\vert w\right\vert }{p}\right) !}.
\end{eqnarray*}%
For $w\in \mathcal{W}$, $\left\vert w\right\vert =0,\dots ,N-\left[ p\right]
-1$,%
\begin{eqnarray*}
&&\left\Vert F^{w}\left( y_{t}^{s,t}\right) -F^{w}\left( y_{s}\right)
-\sum_{l\in \mathcal{W},\left\vert l\right\vert =1,\dots ,\left[ p\right]
}F^{lw}\left( y_{s}\right) \left( T_{s,t}^{X},l\right) \right\Vert \\
&\leq &C\left( p,d,\omega \left( 0,T\right) \right) \left( \left\vert
w\right\vert +\left[ p\right] \right) !\tilde{\omega}\left( s,t\right) ^{%
\frac{\left[ p\right] +1}{p}}.
\end{eqnarray*}
\end{lemma}

\begin{proof}
We prove the first estimate. The proof for the second estimate is similar.
Recall that $T_{n}$ denotes the number of elements in $\mathcal{W}$ of
degree $n$. Based on Lemma \ref{Lemma expression of Taylor expansion}, Lemma %
\ref{Lemma for bound of composition of vector fields}, Lemma \ref{Lemma
uniform bound on yst}, Lemma \ref{Lemma bound 1-var of geodesic by omega}
and that $T_{n}\leq \left( K_{p}d\right) ^{n}$ in Lemma \ref{Lemma bound on
Tn}, we have%
\begin{eqnarray*}
&&\left\Vert F^{w}\left( y_{t}^{s,t}\right) -F^{w}\left( y_{s}\right)
-\sum_{l\in \mathcal{W},\left\vert l\right\vert =1,\dots ,N-\left\vert
w\right\vert }F^{lw}\left( y_{s}\right) \left( T_{s,t}^{X},l\right)
\right\Vert \\
&\leq &C\left( p,d,\omega \left( 0,T\right) \right) N!\left( T_{N-\left\vert
w\right\vert }\omega \left( s,t\right) ^{\frac{N+1-\left\vert w\right\vert }{%
p}}+\sum_{j=1}^{\left[ p\right] -1}T_{N-\left\vert w\right\vert +j}\omega
\left( s,t\right) ^{\frac{N-\left\vert w\right\vert +j}{p}}\right) \\
&\leq &C\left( p,d,\omega \left( 0,T\right) \right) N!\left( \left(
K_{p}d\right) ^{p}\omega \left( s,t\right) \right) ^{\frac{N+1-\left\vert
w\right\vert }{p}} \\
&\leq &C\left( p,d,\omega \left( 0,T\right) \right) N!\frac{\tilde{\omega}%
\left( s,t\right) ^{\frac{N+1-\left\vert w\right\vert }{p}}}{\beta
_{p}\left( \frac{N+1-\left\vert w\right\vert }{p}\right) !}
\end{eqnarray*}%
as $\frac{N+1-\left\vert w\right\vert }{p}\leq \frac{\left[ p\right] +1}{p}%
\leq 2$, where $\tilde{\omega}:=\left( K_{p}d\right) ^{p}\omega $.
\end{proof}

\begin{proposition}
\label{Proposition taylor expansion y}For integer $N\geq \left[ p\right] $, 
\begin{equation*}
\sum_{l\in \mathcal{W},\left\vert l\right\vert =1}^{N}F^{l}\left(
T_{s,t}^{X},l\right) =\sum_{\tau \in T_{\mathcal{L}}^{N}}f\left( \tau
\right) \frac{\left( X_{s,t},\tau \right) }{\sigma \left( \tau \right) }.
\end{equation*}
\end{proposition}

\begin{proof}
According to\ $T_{s,t}^{X}$ in Notation \ref{Notation TX},%
\begin{equation}
\left( T_{s,t}^{X},t_{1}\cdots t_{k}\right) =\left( \bar{X}_{s,t},t_{1}\ast
\cdots \ast t_{k}\right)  \label{inner relation betwen TX and Xbar}
\end{equation}%
for $t_{i}\in \mathcal{B}_{\mathcal{L}}^{\left[ p\right] },\left\vert
t_{1}\right\vert +\cdots +\left\vert t_{k}\right\vert \leq \left[ p\right] $%
. Then based on the construction of the extension of\ $T^{X}$\ and $\bar{X}$ 
\cite[Theorem 2.2.1]{lyons1998differential}, it can be proved inductively
that $\left( \ref{inner relation betwen TX and Xbar}\right) $ holds for $%
t_{i}\in \mathcal{B}_{\mathcal{L}}^{\left[ p\right] },\left\vert
t_{1}\right\vert +\cdots +\left\vert t_{n}\right\vert \leq N,N\geq \left[ p%
\right] $. Moreover, if there exists $i\in \left\{ 1,\dots ,k\right\} $ such
that $t_{i}\in \mathcal{B}_{\mathcal{L}}\backslash \mathcal{B}_{\mathcal{L}%
}^{\left[ p\right] }$, then $\left( \bar{X}_{s,t},t_{1}\ast \cdots \ast
t_{k}\right) =0$.%
\begin{eqnarray*}
\bar{X}_{s,t} &=&\sum_{t_{i}\in \mathcal{B}_{\mathcal{L}}^{\left[ p\right]
},\left\vert t_{1}\right\vert +\cdots +\left\vert t_{k}\right\vert \leq
N}\left( \bar{X}_{s,t},t_{1}\ast \cdots \ast t_{k}\right) t_{1}\ast \cdots
\ast t_{k} \\
&=&\sum_{\rho \in F_{\mathcal{L}}^{N},\left\vert \rho \right\vert \leq
N}\left( \bar{X}_{s,t},\rho \right) \rho \\
&=&\sum_{\rho \in F_{\mathcal{L}}^{N},\left\vert \rho \right\vert \leq N}%
\frac{\left( X_{s,t},\rho \right) }{\sigma \left( \rho \right) }\rho ,
\end{eqnarray*}%
where the last step is based on Proposition \ref{Proposition transfer
between rough paths in two groups}. Combined with Lemma \ref{Lemma
expression of Ft1...tk}, 
\begin{equation*}
\sum_{l\in \mathcal{W},\left\vert l\right\vert =1}^{N}F^{l}\left(
T_{s,t}^{X},l\right) =\psi _{f}\left( \bar{X}_{s,t}\right) I=\sum_{\tau \in
T_{\mathcal{L}}^{N}}f\left( \tau \right) \frac{\left( X_{s,t},\tau \right) }{%
\sigma \left( \tau \right) }.
\end{equation*}
\end{proof}

\begin{proposition}
\label{Proposition difference between y and y st}%
\begin{equation*}
\left\Vert y_{t}-y_{t}^{s,t}\right\Vert \leq C\left( p,d,\omega \left(
0,T\right) \right) \omega \left( s,t\right) ^{\frac{\left[ p\right] +1}{p}}
\end{equation*}
\end{proposition}

\begin{proof}
Let $F^{w}=I$ with $\left\vert w\right\vert =0$ in the second estimate of
Lemma \ref{Lemma initial estimate induction}, and combine with Proposition %
\ref{Proposition taylor expansion y}, 
\begin{equation*}
\left\Vert y_{t}^{s,t}-y_{s}-\sum_{\tau \in T_{\mathcal{L}}^{\left[ p\right]
}}f\left( \tau \right) \left( y_{s}\right) \frac{\left( X_{s,t},\tau \right) 
}{\sigma \left( \tau \right) }\right\Vert \leq C\left( p,d,\omega \left(
0,T\right) \right) \omega \left( s,t\right) ^{\frac{\left[ p\right] +1}{p}}.
\end{equation*}%
Based on \cite[Lemma 17]{lyons2015theory},%
\begin{equation}
\left\Vert y_{t}-y_{s}-\sum_{\tau \in T_{\mathcal{L}}^{\left[ p\right]
}}f\left( \tau \right) \left( y_{s}\right) \frac{\left( X_{s,t},\tau \right) 
}{\sigma \left( \tau \right) }\right\Vert \leq C\left( p,d,\omega \left(
0,T\right) \right) \omega \left( s,t\right) ^{\frac{\left[ p\right] +1}{p}}.
\label{inner Taylor estimate of y}
\end{equation}%
The estimate $\left( \ref{inner Taylor estimate of y}\right) $ can be proved
based on the uniform bound on Picard series \cite[Definition 9, Lemma 17]%
{lyons2015theory} and that Picard series converges to the unique solution 
\cite[Theorem 22]{lyons2015theory}.
\end{proof}

\begin{lemma}
Set $\tilde{\omega}:=\left( K_{p}d\right) ^{p}\omega $. For $w\in \mathcal{W}
$, $\left\vert w\right\vert =N-\left[ p\right] ,\dots ,N$,%
\begin{eqnarray}
&&\left\Vert F^{w}\left( y_{t}\right) -F^{w}\left( y_{s}\right) -\sum_{l\in 
\mathcal{W},\left\vert l\right\vert =1,\dots ,N-\left\vert w\right\vert
}F^{lw}\left( y_{s}\right) \left( T_{s,t}^{X},l\right) \right\Vert
\label{first statement in Lemma needed} \\
&\leq &C\left( p,d,\omega \left( 0,T\right) \right) N!\frac{\left( \tilde{%
\omega}\left( s,t\right) \right) ^{\frac{N+1-\left\vert w\right\vert }{p}}}{%
\beta _{p}\left( \frac{N+1-\left\vert w\right\vert }{p}\right) !}.  \notag
\end{eqnarray}%
For $w\in \mathcal{W}$, $\left\vert w\right\vert =0,\dots ,N-\left[ p\right]
-1$,%
\begin{eqnarray}
&&\left\Vert F^{w}\left( y_{t}\right) -F^{w}\left( y_{s}\right) -\sum_{l\in 
\mathcal{W},\left\vert l\right\vert =1,\dots ,\left[ p\right] }F^{lw}\left(
y_{s}\right) \left( T_{s,t}^{X},l\right) \right\Vert
\label{second statement in Lemma needed} \\
&\leq &C\left( p,d,\omega \left( 0,T\right) \right) \left( \left\vert
w\right\vert +\left[ p\right] \right) !\tilde{\omega}\left( s,t\right) ^{%
\frac{\left[ p\right] +1}{p}}.  \notag
\end{eqnarray}
\end{lemma}

\begin{proof}
Combine Lemma \ref{Lemma for bound of composition of vector fields} with
Proposition\ \ref{Proposition difference between y and y st},%
\begin{equation*}
\left\Vert F^{w}\left( y_{t}^{s,t}\right) -F^{w}\left( y_{t}\right)
\right\Vert \leq \left\vert w\right\vert !\left\Vert
y_{t}^{s,t}-y_{t}\right\Vert \leq C\left( p,d,\omega \left( 0,T\right)
\right) \left\vert w\right\vert !\tilde{\omega}\left( s,t\right) ^{\frac{%
\left[ p\right] +1}{p}}.
\end{equation*}%
When $\left\vert w\right\vert \leq N-1$, the results follow from Lemma \ref%
{Lemma initial estimate induction}. When $\left\vert w\right\vert =N$, based
on Lemma \ref{Lemma for bound of composition of vector fields}, 
\begin{eqnarray*}
\left\Vert F^{w}\left( y_{t}\right) -F^{w}\left( y_{s}\right) \right\Vert
&\leq &N!\left\Vert y_{t}-y_{s}\right\Vert \leq N!\left( \left\Vert
y_{t}-y_{t}^{s,t}\right\Vert +\left\Vert y_{t}^{s,t}-y_{s}\right\Vert \right)
\\
&\leq &C\left( p,d,\omega \left( 0,T\right) \right) N!\frac{\tilde{\omega}%
\left( s,t\right) ^{\frac{1}{p}}}{\beta _{p}\left( \frac{1}{p}\right) !}
\end{eqnarray*}%
where the last step follows from Proposition \ref{Proposition difference
between y and y st} and Lemma \ref{Lemma uniform bound on yst}.
\end{proof}

\begin{lemma}
\label{Lemma factorial decay of TX}For $l\in \mathcal{W},\left\vert
l\right\vert =1,2,\dots $%
\begin{equation*}
\left\Vert \left( T_{s,t}^{X},l\right) \right\Vert \leq \frac{\tilde{\omega}%
\left( s,t\right) ^{\frac{\left\vert l\right\vert }{p}}}{\beta _{p}\left( 
\frac{l}{p}\right) !}
\end{equation*}%
where $\tilde{\omega}=c_{p,d}\omega $ for some constant $c_{p,d}$ depending
on $p,d$.
\end{lemma}

\begin{proof}
Define two norms on $\mathcal{G}_{\mathcal{L}}^{\left[ p\right] }$ as $%
\left\Vert a\right\Vert _{1}:=\max_{\rho \in F_{\mathcal{L}}^{\left[ p\right]
}}\left\vert \left( a,\rho \right) \right\vert ^{\frac{1}{\left\vert \rho
\right\vert }}$ and 
\begin{equation*}
\left\Vert a\right\Vert _{2}:=\max_{t_{i}\in \mathcal{B}_{\mathcal{L}}^{%
\left[ p\right] },\left\vert t_{1}\right\vert +\cdots +\left\vert
t_{n}\right\vert \leq \left[ p\right] }\left\vert \left( a,t_{1}\ast \cdots
\ast t_{n}\right) \right\vert ^{\frac{1}{\left\vert t_{1}\right\vert +\cdots
+\left\vert t_{n}\right\vert }}.
\end{equation*}%
Based on the definition of $T^{X}$ in Notation \ref{Notation TX},
equivalency of continuous homogeneous norms on $\mathcal{G}_{\mathcal{L}}^{%
\left[ p\right] }$ as in Proposition \ref{Proposition equivalency of norms}
and $\left( \bar{X}_{s,t},\rho \right) =\left( X_{s,t},\rho \right) /\sigma
\left( \rho \right) $, we have, for $l\in \mathcal{W}$, $\left\vert
l\right\vert =1,\dots ,\left[ p\right] $, with $\left\Vert
X_{s,t}\right\Vert $ defined at $\left( \ref{norm on GLN}\right) $, 
\begin{equation*}
\left\Vert \left( T_{s,t}^{X},l\right) \right\Vert ^{\frac{1}{\left\vert
l\right\vert }}\leq \left\Vert \bar{X}_{s,t}\right\Vert _{2}\leq
c_{p,d}^{1}\left\Vert \bar{X}_{s,t}\right\Vert _{1}\leq
c_{p,d}^{1}\left\Vert X_{s,t}\right\Vert \leq c_{p,d}^{1}\left\Vert
X\right\Vert _{p-var,\left[ s,t\right] }.
\end{equation*}%
Then the estimate follows from \cite[Theorem 2.2.1]{lyons1998differential}
with $c_{p,d}:=\left( c_{p,d}^{1}\beta _{p}\right) ^{p}$.
\end{proof}

\begin{proof}[Proof of Theorem \protect\ref{main theorem}]
With $K_{p}$ in Lemma \ref{Lemma bound on Tn} and $c_{p,d}$ in Lemma \ref%
{Lemma factorial decay of TX}, denote $c_{p,d}^{2}:=\left( K_{p}d\right)
^{p}\left( c_{p,d}\vee 1\right) $ and set $\tilde{\omega}:=c_{p,d}^{2}\omega 
$. Denote $Y_{t}^{w}:=F^{w}\left( y_{t}\right) $ for $w\in \mathcal{W}%
,\left\vert w\right\vert \leq N$ and $t\in \left[ 0,T\right] $.

Inductive hypothesis: fix $w\in \mathcal{W}$, $\left\vert w\right\vert \leq
N-\left[ p\right] -1$. Suppose for every $w_{1}\in \mathcal{W}$, $\left\vert
w_{1}\right\vert =\left\vert w\right\vert +1,\dots ,N$ and every $s\leq t$,%
\begin{equation*}
\left\Vert Y_{t}^{w_{1}}-Y_{s}^{w_{1}}-\sum_{l\in \mathcal{W},\left\vert
l\right\vert =1}^{N-\left\vert w_{1}\right\vert }Y_{s}^{lw_{1}}\left(
T_{s,t}^{X},l\right) \right\Vert \leq c_{p,d,\omega \left( 0,T\right) }^{1}N!%
\frac{\tilde{\omega}\left( s,t\right) ^{\frac{N+1-\left\vert
w_{1}\right\vert }{p}}}{\beta _{p}\left( \frac{N+1-\left\vert
w_{1}\right\vert }{p}\right) !}.
\end{equation*}%
The statement holds when $\left\vert w\right\vert =N-\left[ p\right] -1$
based on $\left( \ref{first statement in Lemma needed}\right) $.

Denote%
\begin{equation*}
L_{s,t}:=\sum_{l\in \mathcal{W},\left\vert l\right\vert =1}^{N-\left\vert
w\right\vert }Y_{s}^{lw}\left( T_{s,t}^{X},l\right) .
\end{equation*}%
Based on Lemma \ref{Lemma for bound of composition of vector fields}, $%
\left\Vert Y_{s}^{lw}\right\Vert \leq \left( N-1\right) !$ for $l,w\in 
\mathcal{W},\left\vert l\right\vert +\left\vert w\right\vert \leq N$.
Combined with $\left( \ref{second statement in Lemma needed}\right) $ and
Lemma \ref{Lemma factorial decay of TX},%
\begin{eqnarray*}
&&\left\Vert Y_{t}^{w}-Y_{s}^{w}-L_{s,t}\right\Vert \\
&\leq &\left\Vert Y_{t}^{w}-Y_{s}^{w}-\sum_{l\in \mathcal{W},\left\vert
l\right\vert =1}^{\left[ p\right] }Y_{s}^{lw}\left( T_{s,t}^{X},l\right)
\right\Vert +\left\Vert \sum_{l\in \mathcal{W},\left\vert l\right\vert =%
\left[ p\right] +1}^{N-\left\vert w\right\vert }Y_{s}^{lw}\left(
T_{s,t}^{X},l\right) \right\Vert \\
&\leq &C\left( p,d,\omega \left( 0,T\right) ,N\right) \omega \left(
s,t\right) ^{\frac{\left[ p\right] +1}{p}}.
\end{eqnarray*}%
Then%
\begin{equation*}
Y_{t}^{w}-Y_{s}^{w}=\lim_{\left\vert D\right\vert \rightarrow 0,D\subset 
\left[ s,t\right] }\sum_{i,t_{i}\in D}L_{t_{i},t_{i+1}}.
\end{equation*}

For $s\leq u\leq t$,%
\begin{eqnarray*}
&&L_{s,u}+L_{u,t}-L_{s,t} \\
&=&\sum_{l\in \mathcal{W},\left\vert l\right\vert =1}^{N-\left\vert
w\right\vert }Y_{s}^{lw}\left( T_{s,u}^{X},l\right) +\sum_{l\in \mathcal{W}%
,\left\vert l\right\vert =1}^{N-\left\vert w\right\vert }Y_{u}^{lw}\left(
T_{u,t}^{X},l\right) -\sum_{l\in \mathcal{W},\left\vert l\right\vert
=1}^{N-\left\vert w\right\vert }Y_{s}^{lw}\left( T_{s,t}^{X},l\right) \\
&=&\sum_{l\in \mathcal{W},\left\vert l\right\vert =1}^{N-\left\vert
w\right\vert }\left( Y_{u}^{lw}-\sum_{l_{1}\in \mathcal{W},\left\vert
l_{1}\right\vert =0}^{N-\left\vert w\right\vert -\left\vert l\right\vert
}Y_{s}^{l_{1}lw}\left( T_{s,u}^{X},l_{1}\right) \right) \left(
T_{u,t}^{X},l\right) .
\end{eqnarray*}%
Combine the inductive hypothesis and Lemma \ref{Lemma factorial decay of TX}%
, 
\begin{eqnarray*}
&&\left\Vert L_{s,u}+L_{u,t}-L_{s,t}\right\Vert \\
&\leq &c_{p,d,\omega \left( 0,T\right) }^{1}N!\sum_{n=1}^{N-\left\vert
w\right\vert }T_{n}\frac{\tilde{\omega}\left( s,u\right) ^{\frac{%
N+1-n-\left\vert w\right\vert }{p}}}{\beta _{p}\left( \frac{N+1-n-\left\vert
w\right\vert }{p}\right) !}\frac{\left( c_{p,d}\omega \left( u,t\right)
\right) ^{\frac{n}{p}}}{\beta _{p}\left( \frac{n}{p}\right) !}.
\end{eqnarray*}%
where $T_{n}$ denotes the number of elements in $\mathcal{W}$ of order $n$,
and $T_{n}\leq \left( K_{p}d\right) ^{n}$ based on Lemma \ref{Lemma bound on
Tn}.

Since $\tilde{\omega}=\left( K_{p}d\right) ^{p}\left( c_{p,d}\vee 1\right)
\omega $, based on the neo-classical inequality \cite{lyons1998differential,
hara2010fractional},%
\begin{equation*}
\left\Vert L_{s,u}+L_{u,t}-L_{s,t}\right\Vert \leq c_{p,d,\omega \left(
0,T\right) }^{1}N!\frac{p^{2}}{\left( \beta _{p}\right) ^{2}}\frac{\tilde{%
\omega}\left( s,t\right) ^{\frac{N+1-\left\vert w\right\vert }{p}}}{\left( 
\frac{N+1-\left\vert w\right\vert }{p}\right) !}
\end{equation*}%
Since $\left\vert w\right\vert \leq N-\left[ p\right] -1$, $\frac{%
N+1-\left\vert w\right\vert }{p}>\frac{\left[ p\right] +1}{p}$. Successively
dropping points similarly to the proof of \cite[Theorem 2.2.1]%
{lyons1998differential}, 
\begin{equation*}
\left\Vert Y_{t}^{w}-Y_{s}^{w}-\sum_{l\in \mathcal{W},\left\vert
l\right\vert =1}^{N-\left\vert w\right\vert }Y_{s}^{lw}\left(
T_{s,t}^{X},l\right) \right\Vert \leq c_{p,d,\omega \left( 0,T\right) }^{1}N!%
\frac{\tilde{\omega}\left( s,t\right) ^{\frac{N+1-\left\vert w\right\vert }{p%
}}}{\beta _{p}\left( \frac{N+1-\left\vert w\right\vert }{p}\right) !}.
\end{equation*}%
The induction is complete.

Let $w$ be the empty sequence. Then $\left\vert w\right\vert =0$ and $%
Y_{t}^{w}=y_{t}$. Combined with Proposition \ref{Proposition taylor
expansion y}, the proposed estimate holds.
\end{proof}

\bibliographystyle{unsrt}
\bibliography{acompat,roughpath}

\end{document}